\documentclass[12pt,reqno]{amsart}
\usepackage{mathrsfs}
\usepackage{}
\usepackage{amsmath,amsthm,amssymb}
\usepackage{amsfonts}
\usepackage{a4wide}
\allowdisplaybreaks
\numberwithin{equation}{section}

\newtheorem{theorem}{Theorem}[section]

\newtheorem{lemma}[theorem]{Lemma}

\theoremstyle{definition}

\newtheorem{remark}[theorem]{Remark}

\newcommand{\ds}{\displaystyle}

\newcommand{\la}{\lambda}

\newcommand{\al}{\alpha}

\newcommand{\R}{{\mathbb R}}

\begin{document}
\title
[Mass minimizers and concentration] {Mass minimizers and concentration for nonlinear Choquard equations in $\R^N$}\thanks {a: Partially  supported  by  NSFC No: 11371159 }
\author{ Hong yu Ye}

\address{$^*$ H. Y. Ye, College of Science, Wuhan University of Science and Technology, Wuhan, 430065, P. R. China}

\email{ yyeehongyu@163.com}

\begin{abstract}
In this paper, we study the existence of minimizers to the following functional related to the nonlinear Choquard equation:
$$
 E(u)=\frac{1}{2}\ds\int_{\R^N}|\nabla u|^2+\frac{1}{2}\ds\int_{\R^N}V(x)|u|^2-\frac{1}{2p}\ds\int_{\R^N}(I_\al*|u|^p)|u|^p
$$
on $\widetilde{S}(c)=\{u\in H^1(\R^N)|\ \int_{\R^N}V(x)|u|^2<+\infty,\ |u|_2=c,c>0\},$ where $N\geq1$ $\al\in(0,N)$, $\frac{N+\alpha}{N}\leq p<\frac{N+\alpha}{(N-2)_+}$ and $I_\al:\R^N\rightarrow\R$ is the Riesz potential. We present sharp existence results for $E(u)$ constrained on $\widetilde{S}(c)$ when $V(x)\equiv0$ for all $\frac{N+\alpha}{N}\leq p<\frac{N+\alpha}{(N-2)_+}$. For the mass critical case $p=\frac{N+\alpha+2}{N}$, we show that if $0\leq V(x)\in L_{loc}^{\infty}(\R^N)$ and $\lim\limits_{|x|\rightarrow+\infty}V(x)=+\infty$, then mass minimizers exist only if $0<c<c_*=|Q|_2$ and concentrate at the flattest minimum of $V$ as $c$ approaches $c_*$ from below, where $Q$ is a groundstate solution of $-\Delta u+u=(I_\alpha*|u|^{\frac{N+\alpha+2}{N}})|u|^{\frac{N+\alpha+2}{N}-2}u$ in $\R^N$.  \\

\noindent{\bf Keywords:} Choquard equation; Mass concentration; Normalized solutions; Sharp existence.\\
\noindent{\bf Mathematics Subject Classification(2010):} 35J60, 35Q40, 46N50\\
\end{abstract}

\maketitle
\section{Introduction}

In this paper, we consider the following semilinear Choquard problem
\begin{equation}\label{1.1}
  -\Delta u-\mu u=(I_\al*|u|^{p})|u|^{p-2}u,\ \ \  x\in\R^N,\ \mu\in\R
\end{equation}
where $N\geq 1,$ $\al\in(0,N)$, $\frac{N+\alpha}{N}\leq p<\frac{N+\alpha}{(N-2)_+}$, here $\frac{N+\alpha}{(N-2)_+}=\frac{N+\alpha}{N-2}$ if $N\geq3$ and $\frac{N+\alpha}{(N-2)_+}=+\infty$ if $N=1,2$. $I_\al:\R^N\rightarrow\R$ is the Riesz potential \cite{r} defined as $$I_\al(x)=\frac{\Gamma(\frac{N-\al}{2})}{\Gamma(\frac{\alpha}{2})\pi^{\frac{N}{2}}2^\al}\frac{1}{|x|^{N-\al}},\ \ \ \forall\ x\in\R^N\backslash\{0\}.$$

Problem \eqref{1.1} is a nonlocal one due to the existence of the nonlocal nonlinearity. It arises in various fields of mathematical physics, such as quantum mechanics, physics of laser beams, the physics of multiple-particle systems, etc. When $N=3$, $\mu=-1$ and $\al=p=2$, \eqref{1.1} turns to be the well-known Choquard-Pekar equation:
\begin{equation}\label{1.2}
-\Delta u+u=(I_2*|u|^2)u,\ \ \ \   x\in\R^3,
\end{equation}
which was proposed as early as in 1954 by Pekar \cite{pekar}, and by a work of Choquard 1976 in a certain approximation to Hartree-Fock theory for one-component plasma, see \cite{lieb,ls}. \eqref{1.1} is also known as the nonlinear stationary Hartree equation since if $u$ solves \eqref{1.1} then $\psi(t,x)=e^{it}u(x)$ is a solitary wave of the following time-dependent Hartree equation
$$i\psi_t=-\Delta \psi-(I_\alpha*|\psi|^p)|\psi|^{p-2}\psi\ \ \ \hbox{in}\ \R^+\times\R^N,$$ see \cite{gv,mpt}.

In the past years, there are several approaches to construct nontrivial solutions of \eqref{1.1}, see e.g. \cite{csv,lieb,lions,l,m,mpt,tm} for $p=2$ and \cite{ms1,ms2}. One of them is to look for a constrained critical point of the functional
\begin{equation}\label{1.3}
I_p(u)=\frac12\ds\int_{\R^N}|\nabla u|^2-\frac{1}{2p}\ds\int_{\R^N}(I_\alpha*|u|^p)|u|^p
\end{equation}
on the constrained $L^2$-spheres in $H^1(\R^N)$:
$$S(c)=\{u\in H^1(\R^N)|\ |u|_2=c,c>0\}.$$
In this way, the parameter $\mu\in\R$ will appear as a Lagrange multiplier and such solution is called a normalized solution.
By the following well known Hardy-Littlewood-Sobolev inequality: For $1<r,s<+\infty$, if $f\in L^r(\R^N),$ $g\in L^s(\R^N)$, $\lambda\in(0,N)$ and $\frac{1}{r}+\frac1s+\frac{\la}{N}=2$, then
\begin{equation}\label{1.20}
\ds\int_{\R^N}\ds\int_{\R^N}\frac{f(x)g(y)}{|x-y|^\la}\leq C_{r,\la,N}|f|_r|g|_s,
\end{equation}
we see that $I_p(u)$ is well defined and a $C^1$ functional. Set
\begin{equation}\label{1.4}
I_p(c^2)=\inf\limits_{u\in S(c)}I_p(u),
\end{equation}
then minimizers of $I_p(c^2)$ are exactly critical points of $I_p(u)$ constrained on $S(c)$.

 Normalized solutions for equation \eqref{1.2} have been studied in \cite{lieb,lions}. In this paper, one of our purposes is to get a general and sharp result for the existence of minimizers for the minimization problem \eqref{1.4}.

To state our main result, we first prove the following interpolation inequality with the best constant: For $\frac{N+\alpha}{N}<p<\frac{N+\alpha}{(N-2)_+}$,
\begin{equation}\label{1.5}
\ds\int_{\R^N}(I_\alpha*|u|^p)|u|^p \leq \frac{p}{|Q_p|_2^{2p-2}}\left(\ds\int_{\R^N}|\nabla u|^2\right)^{\frac{Np-(N+\alpha)}{2}}\left(\ds\int_{\R^N}|u|^2\right)^{\frac{N+\alpha-(N-2)p}{2}},
\end{equation}
where equality holds for $u=Q_p$, where $Q_p$ is a nontrivial solution of
\begin{equation}\label{1.6}
-\frac{Np-(N+\alpha)}{2}\Delta Q_p+\frac{N+\alpha-(N-2)p}{2}Q_p=(I_\alpha*|Q_p|^p)|Q_p|^{p-2}Q_p,\ \ \ x\in\ \R^N.
\end{equation}
In particular, $Q_{\frac{N+\alpha+2}{N}}$ is a groundstate solution, i.e. the least energy solution among all nontrivial solutions of \eqref{1.6}. Moreover, when $p=\frac{N+\alpha+2}{N}$, all groundstate solutions of \eqref{1.6} have the same $L^2$-norm (see Lemma \ref{lem3.2} below).

Recall in \cite{lieb2} that for $p=\frac{N+\alpha}{N}$, the following Hardy-Littlewood-Sobolev inequality with the best constant:
\begin{equation}\label{1.12}
\ds\int_{\R^N}(I_\alpha*|u|^{\frac{N+\alpha}{N}})|u|^{\frac{N+\alpha}{N}} \leq \frac{1}{|Q_{\frac{N+\al}{N}}|_2^{\frac{2(N+\alpha)}{N}}}\left(\ds\int_{\R^N}|u|^2\right)^{\frac{N+\alpha}{N}}
\end{equation}
with equality if and only if $u=Q_{\frac{N+\al}{N}}$, where $Q_{\frac{N+\al}{N}}=C\left(\frac{\eta}{\eta^2+|x-a|^2}\right)^{\frac{N}{2}},$ $C>0$ is a fixed constant, $a\in\R^N$ and $\eta\in(0,+\infty)$ are parameters.\\

Then our first result is as follows:

\begin{theorem}\label{th1.1}~~Assume that $N\geq1$, $\alpha\in(0,N)$ and $\frac{N+\alpha}{N}\leq p<\frac{N+\alpha}{(N-2)_+}$.

$(1)$~~If $p=\frac{N+\al}{N}$, for any $c>0$,$$I_{\frac{N+\al}{N}}(c^2)=-\frac{N}{2(N+\al)}(\frac{c}{|Q_{\frac{N+\al}{N}}|_2})^{\frac{2(N+\al)}{N}}$$ and $I_{\frac{N+\al}{N}}(c^2)$ has no minimizer.

$(2)$~~If $\frac{N+\alpha}{N}<p<\frac{N+\alpha+2}{N}$, then $I_p(c^2)<0$ for all $c>0$, moreover, $I_p(c^2)$ has at least one minimizer for each $c>0$.

$(3)$~~If $p=\frac{N+\alpha+2}{N}$, let $c_*:=|Q_{\frac{N+\alpha+2}{N}}|_2$, then

~~~~~~~~~~~~\ \ \ (i)~~~$I_{\frac{N+\alpha+2}{N}}(c^2)=\left\{%
\begin{array}{ll}
    0, & \hbox{if~$0<c\leq c_*$}, \\
    -\infty, & \hbox{if~$c>c_*$};\\
\end{array}%
\right.$

~~~~~~~~~~~~\ \ \ (ii)~~$I_{\frac{N+\alpha+2}{N}}(c^2)$ has no minimizer if $c\neq c_*$;

~~~~~~~~~~~~\ \ \ (iii)~~each groundstate of \eqref{1.6} is a minimizer of $I_{\frac{N+\alpha+2}{N}}(c_*^2)$.

~~~~~~~~~~~~\ \ \ (iv)~~there is no critical point for $I_{\frac{N+\alpha+2}{N}}(u)$ constrained on $S(c)$ for each $0<c<c_*$.

$(4)$~~If $\frac{N+\alpha+2}{N}<p<\frac{N+\alpha}{(N-2)_+}$, then $I_p(c^2)$ has no minimizer for each $c>0$ and $I_p(c^2)=-\infty$.
\end{theorem}

\begin{remark}\label{rem1.2}~~Theorem \ref{th1.1} can be seemed as a consequence of the results in Theorem 9 of \cite{lieb} for $p=2$ and in Theorem 1 of \cite{ms1}. However, we still state and prove Theorem \ref{th1.1} here by using an alternative method since our result is delicate and it provides a framework to our subsequent main considerations.
\end{remark}

\begin{remark}\label{rem1.3}~~

$(1)$\ \ $c_*$ is unique.

(2)\ \ Since the positive solution of \eqref{1.6} with $\alpha=p=2$ is uniquely determined up to translations see e.g. \cite{cg,klr,ll}, it follows that if $N=4$ and $\alpha=2$, then up to translations, \textbf{the minimizer of $I_{\frac{N+\alpha+2}{N}}(c_*^2)$ is unique  and  there exists no critical point for $I_{\frac{N+\alpha+2}{N}}(u)$ constrained on ${S(c)}$ for each $c\neq c_*$.}

(3)\ \ For $N\geq3$ and $\frac{N+\alpha+2}{N}<p<\frac{N+\alpha}{N-2}$, it has been proved in \cite{ly} that for each $c>0$, $I_p(u)$ has a mountain pass geometry on $S(c)$ and there exits a couple $(u_c,\mu_c)\in S(c)\times\R^-$ solution of \eqref{1.1} with $I_p(u_c)=\gamma(c)$, where $\gamma(c)$ denotes the mountain pass level on $S(c)$.
\end{remark}

By Theorem 1.1, $p=\frac{N+\alpha+2}{N}$ is called $L^2$-critical exponent for \eqref{1.4}. In order to get critical points under the mass constraint for such $L^2$-critical case, we add a nonnegative perturbation term to the right hand side of \eqref{1.3}, i.e. considering the following functional:
\begin{equation}\label{1.7}
E(u)=\frac{1}{2}\ds\int_{\R^N}|\nabla u|^2+\frac{1}{2}\ds\int_{\R^N}V(x)|u|^2-\frac{N}{2(N+\alpha+2)}\ds\int_{\R^N}(I_\alpha*|u|^{\frac{N+\alpha+2}{N}})|u|^{\frac{N+\alpha+2}{N}},
\end{equation}
where
$$V(x)\in L^{\infty}_{loc}(\R^N),\ \ \inf\limits_{x\in\R^N}V(x)=0\ \ \hbox{and}\ \ \lim\limits_{|x|\rightarrow+\infty}V(x)=+\infty.\eqno{(V_0)}$$
Based on $(V_0)$, we introduce a Sobolev space $\mathcal{H}=\{u\in H^1(\R^N)|\ \int_{\R^N}V(x)|u^2<+\infty\}$ with its associated norm $\|u\|_{\mathcal{H}}=(\int_{\R^N}(|\nabla u|^2+|u|^2+V(x)|u|^2))^{\frac12}.$

\begin{theorem}\label{th1.2}\ \ Assume that $N\geq1$, $\alpha\in(0,N)$ and $(V_0)$ holds. Set
\begin{equation}\label{1.8}
e_c=\inf\limits_{u\in \widetilde{S}(c)}E(u),
\end{equation}
where $\widetilde{S}(c)=\{u\in \mathcal{H}|\ |u|_2=c\}.$ Let $c_*$ be given in Theorem \ref{th1.1}.

(1)\ \ If $0<c<c_*$, then $e_c$ has at least one minimizer and $e_c>0$;

(2)\ \ Let $N-2\leq \alpha<N$ if $N\geq3$ and $0<\al<N$ if $N=1,2$, then for each $c\geq c_*$, $e_c$ has no minimizer; Moreover, $e_c=\left\{%
\begin{array}{ll}
    0, & \hbox{if~$c=c_*$} \\
    -\infty, & \hbox{if~$c>c_*$}\\
\end{array}%
\right.$ and $\lim\limits_{c\rightarrow (c_*)^-}e_c=e_{c_*}.$
\end{theorem}
We also concern the concentration phenomena of minimizers of $e_c$ as $c$ converges to $c_*$ from below. Let $u_c$ be a minimizer of $e_c$ for each $0<c<c_*$, then by \eqref{1.5} and Theorem \ref{th1.2}, we see that
$\int_{\R^N}V(x)|u_c|^2\rightarrow0$ as $c\rightarrow (c_*)^-$, i.e. $u_c$ can be expected to concentrate at the minimum of $V(x)$.
To show this fact, besides condition $(V_0)$, we assume that there exist $m\geq1$ distinct points $x_i\in\R^N$ and $q_i>0$ $(1\leq i\leq m)$ such that
$$
\mu_i:=\lim\limits_{x\rightarrow x_i}\frac{V(x)}{|x-x_i|^{q_i}}\in (0,+\infty).\eqno{(V_1)}
$$
Set
$$q:=\max\{q_1,q_2,\cdots,q_m\}.$$
Let $\{c_k\}\subset(0,c_*)$ be a sequence such that $c_k\rightarrow c_*$ as $k\rightarrow+\infty$. Then Our main result is as follows:

\begin{theorem}\label{th1.3}\ \ Suppose that $N\geq1$, $\alpha\in[N-2,N)$ if $N\geq3$ and $\al\in(0,N)$ if $N=1,2$ and $(V_0)(V_1)$ hold. Then there exists a sequence $\{x_k\}\subset\R^N$ and a groundstate solution $W_0$ of the following equation
\begin{equation}\label{1.9}
-\Delta W_0+W_0=(I_\alpha*|W_0|^{\frac{N+\alpha+2}{N}})|W_0|^{\frac{N+\alpha+2}{N}-2}W_0,\ \ \ x\in\R^N
\end{equation}
and
$$\lambda:=\min\limits_{1\leq i\leq m}\left\{\lambda_i|\ \lambda_i=\left(\frac{q_i}{2c_*^2}\mu_i\ds\int_{\R^N}|x|^{q_i}|W_0(x)|^2\right)^{\frac{1}{q_i+2}}\right\}$$
such that up to a subsequence,
\begin{equation}\label{1.10} [1-(\frac{c_k}{c_*})^{\frac{2(\alpha+2)}{N}}]^{\frac{1}{q+2}\frac{N}{2}}u_{c_k}([1-(\frac{c_k}{c_*})^{\frac{2(\alpha+2)}{N}}]^{\frac{1}{q+2}} x+x_k)\rightarrow[(\frac{\alpha+2}{N})^{\frac{1}{q+2}}\lambda]^{\frac{N}{2}} W_0((\frac{\alpha+2}{N})^{\frac{1}{q+2}}\lambda x)
\end{equation}
in $L^{\frac{2Ns}{N+\alpha}}(\R^N)$ for $\frac{N+\alpha}{N}\leq s<\frac{N+\alpha}{(N-2)+}$ as $k\rightarrow+\infty$.
Moreover, there exists $x_{j_0}\in \{x_i|~\lambda_i=\lambda,~1\leq i\leq m\}$ such that $x_k\rightarrow x_{j_0}$ as $k\rightarrow+\infty$.
\end{theorem}

\begin{remark}\label{rem1.1}\ \ It has been proved in \cite{ms1} that for $\alpha\in[N-2,N)$ if $N\geq3$ and $\al\in(0,N)$ if $N=1,2$, then each groundstate solution $u$ of \eqref{1.9} satisfies that $\lim\limits_{|x|\rightarrow+\infty} |u(x)||x|^{\frac{N-1}{2}}e^{|x|}\in(0,+\infty)$. Hence $\lambda_i\in(0,+\infty)$.
\end{remark}
The result in Theorem \ref{th1.3} is different from that in \cite{me} studying the case $p<\frac{N+\alpha+2}{N}$, where one considered the concentration behavior of minimizers as $c\rightarrow+\infty$. The concentration phenomena have also been studied in \cite{ms3} and \cite{css} by considering semiclassical limit of the Choquard equation
$$-\varepsilon^2\Delta u+Vu=\varepsilon^{-\alpha}(I_\alpha*|u|^p)|u|^{p-2}u~~~~\ \ \ \hbox{in}~~\R^N.$$
However, since the parameter is different, we need a different technique to obtain our result.

The main proof of Theorem \ref{th1.3} is based on optimal energy estimates of $e_c$ and $\int_{\R^N}|\nabla u_c|^2$ for each minimizer $u_c$. The main idea to prove Theorem \ref{th1.3} comes from \cite{gs}, which was restricted to the case of local nonlinearities. But due to the fact that our nonlinearity is nonlocal and that the assumption imposed on $(V)$ is more general than that in \cite{gs}, the method used in \cite{gs} can not be directly applied here. It needs some improvements and careful analysis. First, by choosing a suitable test function, we get that $0<e_c\leq C_1[1-(\frac{c}{c_*})^{\frac{2(\alpha+2)}{N}}]^{\frac{q}{q+2}}$ as $c\rightarrow (c_*)^-$ for some constant $C_1>0$ independent of $c$. The lower bound now is not optimal. The method in \cite{gs} by using the perturbation term $\int_{\R^N}V(x)u^2$ to remove the local nonlinearity term does not work in our cases. To obtain an optimal lower bound, we notice that $\int_{\R^N}|\nabla u_c|^2\rightarrow+\infty$ as $c\rightarrow (c_*)^-$, moreover,
$$\lim\limits_{c\rightarrow (c_*)^-}\frac{\frac{N}{N+\alpha+2}\int_{\R^N}(I_\alpha*|u_c|^{\frac{N+\alpha+2}{N}})|u_c|^{\frac{N+\alpha+2}{N}}}{\int_{\R^N}|\nabla u_c|^2}=1.$$
Then by taking a special $L^2$-preserving scaling as:
\begin{equation}\label{1.11}
w_c(x)=\varepsilon_c^{\frac{N}{2}}u_c(\varepsilon_c x+\varepsilon_cy_c),
\end{equation}
where
$$\varepsilon^2_c=\frac{2(N+\alpha+2)}{N\int_{\R^N}(I_\alpha*|u_c|^{\frac{N+\alpha+2}{N}})|u_c|^{\frac{N+\alpha+2}{N}}}\rightarrow0\ \ \ \hbox{as}\ c\rightarrow (c_*)^-$$
and the sequence $\{y_c\}$ is derived from the vanishing lemma, we succeeded in proving that there is a constant $C_2>0$ independent of $c$ such that $$\ds\int_{\R^N}V(\varepsilon_cx+\varepsilon_cy_c)|w_c(x)|^2\geq C_2\varepsilon_c^q\ \ \ \hbox{as}\ c\rightarrow (c_*)^-,$$
which and \eqref{1.5} implies that
$e_c\geq C_3[1-(\frac{c}{c_*})^{\frac{2(\alpha+2)}{N}}]^{\frac{q}{q+2}}$ for some constant $C_3>0$ independent of $c$. In succession, there exist two constants $0<C_4<C_5$ independent of $c$ such that $C_4[1-(\frac{c}{c_*})^{\frac{2(\alpha+2)}{N}}]^{\frac{-2}{q+2}}\leq \int_{\R^N}|\nabla u_c|^2\leq C_5[1-(\frac{c}{c_*})^{\frac{2(\alpha+2)}{N}}]^{\frac{-2}{q+2}}$. Finally, by using the Euler-Lagrange equation $u_c$ satisfied and the scaling \eqref{1.11} again with $\varepsilon_c=[1-(\frac{c}{c_*})^{\frac{2(\alpha+2)}{N}}]^{\frac{1}{q+2}}$, we show that
$$e_c\approx[1-(\frac{c}{c_*})^{\frac{2(\alpha+2)}{N}}]^{\frac{q}{q+2}}\frac{q+2}{q}\frac{\lambda^2c_*^2}{2}\left(\frac{N}{\alpha+2}\right)^{\frac{q}{q+2}}\ \ \hbox{as}\ c\rightarrow (c_*)^-,$$
which implies \eqref{1.10}.

Throughout this paper, we use standard notations. For simplicity, we
write $\int_{\Omega} h$ to mean the Lebesgue integral of $h(x)$ over
a domain $\Omega\subset\R^N$. $L^{p}:= L^{p}(\R^{N})~(1\leq
p<+\infty)$ is the usual Lebesgue space with the standard norm
$|\cdot|_{p}.$ We use `` $\rightarrow"$ and `` $\rightharpoonup"$ to denote the
strong and weak convergence in the related function space
respectively. $C$ will
denote a positive constant unless specified. We use `` $:="$ to denote definitions. We denote a subsequence
of a sequence $\{u_n\}$ as $\{u_n\}$ to simplify the notation unless
specified.

The paper is organized as follows. In Section 2,  we will determine the best constant for the interpolation estimate \eqref{1.5} and give the proof of Theorem \ref{th1.1}. In section 3, we prove Theorems \ref{th1.2} and \ref{th1.3}.

%\begin{center}
\section{ Proof of Theorem \ref{th1.1}}
%\end{center}
In this section, we first prove the interpolation estimate \eqref{1.5}. It is enough to consider the following minimization problem: $$S_p=\inf\limits_{u\in H^1(\R^N)\backslash\{0\}}W_p(u),$$ where
$$W_p(u)=\frac{\left(\int_{\R^N}|\nabla u|^2\right)^{\frac{Np-(N+\alpha)}{2}}\left(\int_{\R^N}|u|^2\right)^{\frac{N+\alpha-(N-2)p}{2}}}{\int_{\R^N}(I_\alpha*|u|^p)|u|^p}.$$

\begin{lemma}\label{lem2.1}(\cite{ms1}, Lemma 2.4)~~Let $N\geq1$, $\alpha\in (0,N)$, $p\in [1,\frac{2N}{N+\alpha})$ and $\{u_n\}$ be a bounded sequence in $L^{\frac{2Np}{N+\alpha}}(\R^N)$. If $u_n\rightarrow u$ a.e. in $\R^N$ as $n\rightarrow+\infty$, then
$$\lim\limits_{n\rightarrow+\infty}\left(\ds\int_{\R^N}(I_\alpha*|u_n|^p)|u_n|^p-\ds\int_{\R^N}(I_\alpha*|u_n-u|^p)|u_n-u|^p\right)
=\ds\int_{\R^N}(I_\alpha*|u|^p)|u|^p.$$
\end{lemma}

\begin{lemma}\label{lem2.6} (\cite{wi}, Vanishing Lemma)\ \
Let $r>0$ and $2\leq q<2^*$. If $\{u_n\}$ is bounded in
$H^1(\R^N)$ and
$$\sup\limits_{y\in\R^N}\ds\int_{B_r(y)}|u_n|^q\rightarrow0,~~n\rightarrow+\infty,$$
then $u_n\rightarrow0$ in $L^s(\R^N)$ for $2<s<2^*$.
\end{lemma}
\begin{lemma}\label{lem2.2}~~Let $N\geq1$, $\alpha\in(0,N)$ and $\frac{N+\alpha}{N}<p<\frac{N+\alpha}{(N-2)_+}$, then
$S_{p}$ is achieved by a function $Q_p\in H^1(\R^N)\backslash\{0\}$, where $Q_p$ is a nontrivial solution of equation \eqref{1.6}
and $$S_{p}=\frac{|Q_p|_2^{2p-2}}{p}.$$
\end{lemma}

\begin{proof}~~The lemma can be viewed as a consequence of Proposition 2.1 in \cite{ms1} and Theorem 9 in \cite{lieb}, but we give an alternative proof here. The idea of the proof comes from \cite{we}, but some details are delicate.

Since $W_{p}(u)\geq0$ for any $u\in H^1(\R^N)\backslash\{0\}$, $S_{p}$ is well defined. Let $\{u_n\}\subset H^1(\R^N)\backslash\{0\}$ be a minimizing sequence for $S_p$, i.e. $W_{p}(u_n)\rightarrow S_{p}$ as $n\rightarrow+\infty$. Set $$\lambda_n:=\frac{(\int_{\R^N}|u_n|^2)^{\frac{N-2}{4}}}{(\int_{\R^N}|\nabla u_n|^2)^{\frac{N}{4}}},\, \, \, \, \, \, \, \, \, \, \, \, \, \, \, \,  \mu_n:=\frac{(\int_{\R^N}|u_n|^2)^{\frac12}}{(\int_{\R^N}|\nabla u_n|^2)^{\frac12}}$$
and $$v_n(x):=\lambda_nu_n(\mu_nx).$$
Then $\int_{\R^N}|v_n|^2=\int_{\R^N}|\nabla v_n|^2=1$ and \begin{equation}\label{2.1}
W_{p}(v_n)=W_{p}(u_n)\rightarrow S_{p}\ \ \ \hbox{as}\ \ n\rightarrow+\infty,
\end{equation}
i.e. $\{v_n\}$ is a bounded minimizing sequence for $S_{p}$.

Let
$\delta:=\lim\limits_{n\rightarrow+\infty}\sup\limits_{y\in\R^N}\int_{B_1(y)}|v_n|^2.$
If $\delta=0$, then by Lemma \ref{lem2.6}, $v_n\rightarrow 0$ in $L^s(\R^N)$, $2<s<2^*$. Hence by the Hardy-Littlewood-Sobolev inequality \eqref{1.20}, $$W_{p}(v_n)=\frac{1}{\int_{\R^N}(I_\alpha*|v_n|^p)|v_n|^p}\rightarrow+\infty,$$
which contradicts \eqref{2.1}. Therefore, $\delta>0$ and there exists a sequence $\{y_n\}\subset\R^N$ such that \begin{equation}\label{2.2}
\int_{B_1(y_n)}|v_n|^2\geq\frac{\delta}{2}>0.
\end{equation}
Up to translations, we may assume that $y_n=0$. Since $\{v_n\}$ is bounded in $H^1(\R^N)$ and by \eqref{2.2}, there exists $v_{p}\in H^1(\R^N)\backslash\{0\}$ such that $v_n\rightharpoonup v_{p}$ in $H^1(\R^N)$. Then by the Brezis Lemma and Lemma \ref{lem2.1}, we have
$$\begin{array}{ll}
S_{p} &\leq W_{p}(v_{p})\\[5mm]
&\leq\lim\limits_{n\rightarrow+\infty}\ds\left[W_{p}(v_n)\frac{\int_{\R^N}(I_\alpha*|v_n|^p)|v_n|^p}{\int_{\R^N}(I_\alpha*|v_{p}|^p)|v_{p}|^p}-
W_{p}(v_n-v_{p})\frac{\int_{\R^N}(I_\alpha*|v_n-v_{p}|^p)|v_n-v_{p}|^p}{\int_{\R^N}(I_\alpha*|v_p|^p)|v_{p}|^p}\right]\\[5mm]
&\leq S_{p} \lim\limits_{n\rightarrow+\infty}\ds\left(\frac{\int_{\R^N}(I_\alpha*|v_n|^p)|v_n|^p-\int_{\R^N}(I_\alpha*|v_n-v_p|^p)|v_n-v_{p}|^p}{\int_{\R^N}(I_\alpha*|v_p|^p)|v_{p}|^p}\right)\\[5mm]
&=S_{p},
\end{array}$$
i.e. $W_{p}(v_{p})=S_{p}.$ Moreover, $|\nabla v_{p}|_2=|v_{p}|_2=1$ and $S_p=\frac{1}{\int_{\R^N}(I_\alpha*|v_{p}|^p)|v_{p}|^p}.$

Therefore, for any $h\in H^1(\R^N)$, $\frac{d}{dt}\Big{|}_{t=0}W_{p}(v_p+th)=0$, i.e. $v_p$ satisfies the following equation
$$-[Np-(N+\alpha)]\Delta v_p+[N+\alpha-(N-2)p]v_p=2pS_{p}(I_\alpha*|v|^p)|v_p|^{p-2}v_p,~~~\hbox{in}~~\R^N.$$
Let $v_p=(\frac{1}{pS_{p}})^{\frac{1}{2p-2}}Q_p$, then $Q_p$ is a nontrivial solution of \eqref{1.6} and $S_{p}=\frac{|Q_p|_2^{2p-2}}{p}$.\\
\end{proof}

Next we give the proof of Theorem \ref{th1.1}. For any $u\in S(c)$, set
$$A(u):=\ds\int_{\R^N}|\nabla u|^2,\, \, \, \, \, \, \, \, \, \, B(u):=\ds\int_{\R^N}(I_\alpha*|u|^p)|u|^p,$$
then $I_p(u)=\frac{1}{2}A(u)-\frac{1}{2p}B(u).$ It follows from \eqref{1.5}\eqref{1.6} that for $\frac{N+\al}{N}<p<\frac{N+\al}{(N-2)_+}$,
\begin{equation}\label{2.4}
B(u)\leq \frac{p}{|Q_p|_2^{2p-2}}A(u)^{\frac{Np-(N+\alpha)}{2}}c^{N+\alpha-(N-2)p}
\end{equation}
with equality for $u=Q_p$ given in \eqref{1.6}, moreover,
\begin{equation}\label{2.5}
A(Q_p)=\frac{1}{p}B(Q_p)=|Q_p|_2^2.
\end{equation}

\begin{lemma}\label{lem2.3}~~Let $N\geq1$ and $\al\in(0,N)$.

$(1)$\ \ If $\frac{N+\alpha}{N}<p<\frac{N+\alpha+2}{N}$, then $I_p(u)$ is bounded from below and coercive on $S(c)$ for all $c>0$, moreover, $I_{p}(c^2)<0$.

$(2)$\ \ If $p=\frac{N+\alpha+2}{N}$, then $I_{\frac{N+\alpha+2}{N}}(c^2)=\left\{%
\begin{array}{ll}
    0, & 0<c\leq c_*:=|Q_{\frac{N+\alpha+2}{N}}|_2, \\
   -\infty, & c>c_*,\\
\end{array}%
\right.$

$(3)$\ \ If $\frac{N+\alpha+2}{N}<p<\frac{N+\alpha}{(N-2)_+}$, then $I_{p}(c^2)=-\infty$ for all $c>0$.
\end{lemma}
\begin{proof}~~$(1)$~~For any $c>0$ and $u\in S(c)$, by \eqref{2.4}, there exists $C:=\frac{c^{N+\alpha-(N-2)p}}{|Q_p|_2^{2p-2}}$ such that
\begin{equation}\label{2.6}
I_p(u)\geq \frac{A(u)-
CA(u)^{\frac{Np-(N+\alpha)}{2}}}{2}.
\end{equation}
Since $\frac{N+\alpha}{N}<p<\frac{N+\alpha+2}{N}$, $0<Np-(N+\alpha)<2$. Then \eqref{2.6} implies that $I_p(u)$ is bounded from below and coercive on $S(c)$ for any $c>0$.

 Set $u^t(x):=t^{\frac{N}{2}}u(tx)$ with $t>0$, then $u^t\in S(c)$ and
\begin{equation}\label{2.7}
I_p(u^t)=\frac{t^2}{2}A(u)-\frac{t^{Np-(N+\alpha)}}{2p}B(u)<0\ \ \ \ \ \hbox{for}\ t>0\ \hbox{small~enough}\end{equation}
since $0<Np-(N+\alpha)<2$, which implies that $I_{p}(c^2)<0$ for each $c>0$.

$(2)$~~When $p=\frac{N+\alpha+2}{N}$, $Np-(N+\alpha)=2$, similarly to \eqref{2.6} and \eqref{2.7}, we have
$$I_{\frac{N+\alpha+2}{N}}(u)\geq \frac{A(u)}{2}\left[1-\left(\frac{c}{c_*}\right)^{\frac{2(\alpha+2)}{N}}\right]\geq0\ \ \ \ \hbox{if}\ \ \ 0<c\leq c_*$$
and $I_{\frac{N+\alpha+2}{N}}(c^2)\leq I_{\frac{N+\alpha+2}{N}}(u^t)\rightarrow 0$ as $t\rightarrow 0^+$ for all $c$. Then $I_{\frac{N+\alpha+2}{N}}(c^2)=0$ if $0<c\leq c_*$.

If $c>c_*$, set $Q^t(x):=\frac{ct^{\frac{N}{2}}}{c_*}Q_{\frac{N+\alpha+2}{N}}(tx)$, then by \eqref{2.5}, $$I_{\frac{N+\alpha+2}{N}}(Q^t)=\frac{c^2t^2}{2c_*^2}\left[1-\left(\frac{c}{c_*}\right)^{\frac{2(\alpha+2)}{N}}\right]\rightarrow -\infty\ \ \hbox{as}\ \  t\rightarrow+\infty,$$ then $I_{\frac{N+\alpha+2}{N}}(c^2)=-\infty$ for $c>c_*$.

$(3)$~~If $\frac{N+\alpha+2}{N}<p<\frac{N+\alpha}{(N-2)_+}$, then $Np-(N+\alpha)>2$, hence by \eqref{2.7}, we have $I_p(u^t)\rightarrow -\infty$ as $t\rightarrow+\infty$, so $I_{p}(c^2)=-\infty$ for all $c>0$.
\end{proof}

\begin{lemma}\label{lem2.4}~~If $\frac{N+\alpha}{N}<p<\frac{N+\alpha+2}{N}$, then

$(1)$~~the function $c\mapsto I_p(c^2)$ is continuous on $(0,+\infty)$;

$(2)$~~
\begin{equation}\label{2.8}
I_p(c^2)<I_p(\alpha^2)+I_p(c^2-\alpha^2),\ \ \ \ \ \ \forall\ 0<\alpha<c<+\infty.
\end{equation}

\end{lemma}
\begin{proof}~~The proof of $(1)$ follows from Lemma \ref{lem2.3} and is similar to that of Theorem 2.1 in \cite{bs}, so we omit it.

$(2)$~~For any $c>0$, let $\{u_n\}\subset S(c)$ be a minimizing sequence for $I_p(c^2)<0$, then by Lemma \ref{lem2.3}, $\{u_n\}$ is bounded in $H^1(\R^N)$ and there exists a constant $K_1>0$ independent of $n$ such that $B(u_n)\geq K_1$. Set $u_n^\theta=\theta u_n$ with $\theta>1$, then $u_n^\theta\in S(\theta c)$ and
$$I_p(u_n^\theta)-\theta^2 I(u_n)=\frac{\theta^2-\theta^{2p}}{2p}B(u_n)\leq \frac{\theta^2-\theta^{2p}}{2p}K_1<0.$$
Letting $n\rightarrow+\infty$, we have $I_p(\theta^2c^2)<I_p(c^2)$, $\theta>1$, which easily implies \eqref{2.8} by using Lemma \ref{lem2.3} (1).
\end{proof}

\begin{lemma}\label{lem2.5}~~Let $N\geq1$, $\al\in(0,N)$ and $\frac{N+\alpha}{N}<p<\frac{N+\alpha}{(N-2)_+}$. If $u$ is a critical point of $I_p(u)$ constrained on $S(c)$, then there exists $\mu_c<0$ such that $I_p'(u)-\mu_c u=0$ in $H^{-1}(\R^N)$ and
$$A(u)-\frac{Np-(N+\alpha)}{2p}B(u)=0.$$
\end{lemma}
\begin{proof}~~Since $(I_p|_{S(c)})'(u)=0$, there exists $\mu_c\in\R$ such that $I_p'(u)-\mu_cu=0$ in $H^{-1}(\R^N)$. Then
$$A(u)-B(u)=\mu_cc^2.$$
By Proposition 3.5 in \cite{ms2}, $u$ satisfies the following Pohozaev identity,
$$\frac{N-2}{2}A(u)-\frac{N+\alpha}{2p}B(u)=\frac{N}{2}\mu_cc^2.$$
Hence $A(u)=\frac{Np-(N+\alpha)}{2p}B(u)$ and
$$\mu_c=\frac{(N-2)p-(N+\alpha)}{2pc^2}B(u)<0.$$
\end{proof}

\noindent $\textbf{Proof of Theorem 1.1}$\\
\begin{proof}~~(1)~~If $p=\frac{N+\al}{N}$, for any $c>0$ and $u\in S(c)$, by \eqref{1.12} we have
$$I_{\frac{N+\alpha}{N}}(u)\geq-\frac{N}{2(N+\al)}\left(\frac{c}{|Q_{\frac{N+\alpha}{N}}|_2}\right)^{\frac{2(N+\al)}{N}}.$$
Set $Q_{\frac{N+\alpha}{N}}^t(x):=\frac{ct^{\frac N2}}{|Q_{\frac{N+\alpha}{N}}|_2}Q_{\frac{N+\alpha}{N}}(tx)$, then by \eqref{1.12} again, we see that
$$I_{\frac{N+\alpha}{N}}(Q_{\frac{N+\alpha}{N}}^t)=\frac{c^2t^2}{2|Q_{\frac{N+\alpha}{N}}|_2^2}A(Q_{\frac{N+\alpha}{N}})
-\frac{N}{2(N+\al)}\left(\frac{c}{|Q_{\frac{N+\alpha}{N}}|_2}\right)^{\frac{2(N+\al)}{N}},$$
letting $t\rightarrow 0^+$, then $I_{\frac{N+\alpha}{N}}(c^2)=-\frac{N}{2(N+\al)}(\frac{c}{|Q_{\frac{N+\alpha}{N}}|_2})^{\frac{2(N+\al)}{N}}.$

By contradiction, if for some $c>0$, there is $u\in S(c)$ such that $I_{\frac{N+\alpha}{N}}(u)=I_{\frac{N+\alpha}{N}}(c^2)$, then \eqref{1.12} shows that
$$0\leq\frac{1}{2}A(u)=\frac{N}{2(N+\al)}\left[B(u)-\left(\frac{c}{|Q_{\frac{N+\alpha}{N}}|_2}\right)^{\frac{2(N+\al)}{N}}\right]\leq0,$$
which implies that $u=0$. It is a contradiction. So $I_{\frac{N+\alpha}{N}}(c^2)$ has no minimizer for all $c>0$.

(2)~~If $\frac{N+\alpha}{N}<p<\frac{N+\al+2}{N}$, for any $c>0$, by Lemma \ref{lem2.3}, $I_p(c^2)<0$. Let $\{u_n\}\subset S(c)$ be a minimizing sequence for $I_p(c^2)$, then Lemma \ref{lem2.3} (1) implies that $\{u_n\}$ is bounded in $H^1(\R^N)$ and for some constant $C>0$ independent of $n$, $B(u_n)\geq C$. Hence there exists $u\in H^1(\R^N)$ such that
\begin{equation}\label{2.9}
u_n\rightharpoonup u\ \ \hbox{in}\ H^1(\R^N),\ \ \ \ \ \ \ u_n(x)\rightarrow u(x)\ \ \hbox{a.e.}\ \hbox{in}\ \R^N.
\end{equation}
Moreover, by the Vanishing Lemma \ref{lem2.6}, up to translations, we may assume that $u\neq0$. Then $0<|u|_2:=\alpha\leq c$. We just suppose that $\alpha<c$, then $u\in S(\alpha)$. By \eqref{2.9} and the Brezis lemma, we have
$$\lim\limits_{n\rightarrow+\infty}|u_n-u|_2^2=\lim\limits_{n\rightarrow+\infty}|u_n|_2^2-|u|_2^2=c^2-\alpha^2.$$
Then by Lemma \ref{lem2.1} and Lemma \ref{lem2.4} (1), we have
$$
I_p(c^2)=\lim\limits_{n\rightarrow+\infty}I_p(u_n)
=\lim\limits_{n\rightarrow+\infty}I_p(u_n-u)+I_p(u)
\geq I_p(c^2-\alpha^2)+I_p(\alpha^2),$$
which contradicts \eqref{2.8}. So $|u|_2=c$, i.e. $u_n\rightarrow u$ in $L^2(\R^N)$. By \eqref{2.4}, we have $B(u_n)\rightarrow B(u)$. Then
$$I_p(c^2)\leq I_p(u)\leq \lim\limits_{n\rightarrow+\infty}I_p(u_n)=I_p(c^2),$$
i.e. $u$ is minimizer for $I_p(c^2)$.

$(3)$~~(i) has been proved in Lemma \ref{lem2.4} (2). To prove (ii), by contradiction, if there exists $c_0\in (0,c_*)$ such that $I_{\frac{N+\alpha+2}{N}}(c_0^2)$ has a minimizer $u_0\in S(c_0)$, i.e. $I_{\frac{N+\alpha+2}{N}}(u_0)=I_{\frac{N+\alpha+2}{N}}(c_0^2)=0$, then by \eqref{2.4},
$$A(u_0)=\frac{N}{N+\alpha+2}B(u_0)\leq \left(\frac{c_0}{c_*}\right)^{\frac{2(\alpha+2)}{N}}A(u_0)<A(u_0),$$
which is impossible. So combining (i), we see that $I_{\frac{N+\alpha+2}{N}}(c^2)$ has no minimizer for all $c\neq c_*$.

By \eqref{2.5}, we see that $I_{\frac{N+\alpha+2}{N}}(Q_{\frac{N+\alpha+2}{N}})=0=I_{\frac{N+\alpha+2}{N}}(c_*^2),$ i.e. $Q_{\frac{N+\alpha+2}{N}}$ is a minimizer for $I_{\frac{N+\alpha+2}{N}}(c_*^2)$. Moreover, by Lemmas \ref{lem3.1} (2) and \ref{lem3.2} below, each groundstate solution of \eqref{1.6} is a minimizer of $I_{\frac{N+\alpha+2}{N}}(c_*^2)$. So we proved (iii).

For any $c>0$, suppose that $u$ is a critical point of $I_{\frac{N+\alpha+2}{N}}(u)$ constrained on $S(c)$, then by \eqref{2.5} and Lemma \ref{lem2.5}, we have
$$A(u)=\frac{N}{N+\alpha+2}B(u)\leq \left(\frac{c}{c_*}\right)^{\frac{2(\alpha+2)}{N}}A(u),$$
which implies that $c_*\leq c$. Therefore, there exists no critical point for $I_{\frac{N+\alpha+2}{N}}(u)$ constrained on $S(c)$ if $0<c<c_*$. So (iv) is proved.

(4)\ \ By Lemma \ref{lem2.3} (3), $I_{p}(c^2)$ has no minimizer for all $c>0$ if $\frac{N+\al+2}{N}<p<\frac{N+\al}{(N-2)_+}$.

\end{proof}

%\begin{center}
\section{ Proof of Theorems \ref{th1.2} and \ref{th1.3}}
%\end{center}

 For $p=\frac{N+\alpha+2}{N}$, \eqref{2.4} turns to be
\begin{equation}\label{3.4}
B(u)\leq \frac{N+\alpha+2}{N}\left(\frac{1}{c_*}\right)^{\frac{2(\alpha+2)}{N}}A(u)|u|_2^{\frac{2(\alpha+2)}{N}},
\end{equation}
with equality for $u=Q_{\frac{N+\alpha+2}{N}}$ and $c_*:=|Q_{\frac{N+\alpha+2}{N}}|_2$, where $Q_{\frac{N+\alpha+2}{N}}$ is a nontrivial solution of
$$-\Delta Q_{\frac{N+\alpha+2}{N}}+\frac{\alpha+2}{N}Q_{\frac{N+\alpha+2}{N}}=(I_\alpha*|Q_{\frac{N+\alpha+2}{N}}|^{\frac{N+\alpha+2}{N}})|Q_{\frac{N+\alpha+2}{N}}|^{\frac{N+\alpha+2}{N}-2}Q_{\frac{N+\alpha+2}{N}},\ \ \hbox{in}\ \R^N.$$
Set $Q_{\frac{N+\alpha+2}{N}}(x)=\left(\sqrt{\frac{\alpha+2}{N}}\right)^{\frac{N}{2}}\widetilde{Q}_{\frac{N+\alpha+2}{N}}(\sqrt{\frac{\alpha+2}{N}}x)$, then $\widetilde{Q}_{\frac{N+\alpha+2}{N}}$ satisfies the equation
\begin{equation}\label{3.1}
-\Delta \widetilde{Q}_{\frac{N+\alpha+2}{N}}+\widetilde{Q}_{\frac{N+\alpha+2}{N}}
=(I_\alpha*|\widetilde{Q}_{\frac{N+\alpha+2}{N}}|^{\frac{N+\alpha+2}{N}})|\widetilde{Q}_{\frac{N+\alpha+2}{N}}|^{\frac{N+\alpha+2}{N}-2}\widetilde{Q}_{\frac{N+\alpha+2}{N}}, \ \ \hbox{in}\ \R^N.
\end{equation}

The following Lemma is a direct conclusion of Theorems 1-4 in \cite{ms1}.
\begin{lemma}\label{lem3.1}~~Assume that $N\geq1$ and $\alpha\in(0,N)$.

$(1)$~~There is at least one groundstate solution $u\in H^1(\R^N)$ to \eqref{3.1} with
$$F(u)=d:=\inf\{F(v)|~v\in H^1(\R^N)\backslash\{0\}~\hbox{is~a~weak~solution~of}~\eqref{3.1}\},$$
where $F(v)=\ds\frac{1}{2}\int_{\R^N}(|\nabla v|^2+|v|^2)-\frac{N}{2(N+\alpha+2)}\ds\int_{\R^N}(I_\alpha*|v|^{\frac{N+\alpha+2}{N}})|v|^{\frac{N+\alpha+2}{N}}.$

$(2)$~~If $u\in H^1(\R^N)$ is a nontrivial solution of \eqref{3.1}, then $u\in L^1(\R^N)\cap C^2(\R^N),$ $u\in W^{2,s}(\R^N)$ for every $s>1$ and $u\in C^\infty(\R^N\backslash u^{-1}(\{0\})$. Moreover,
\begin{equation}\label{3.2}
\frac{N+\alpha+2}{N}A(u)=\frac{N+\alpha+2}{\alpha+2}
\ds\int_{\R^N}|u|^2
=B(u).
\end{equation}

$(3)$~~If $u$ is a groundstate solution of \eqref{3.1}, then $u$ is either positive or negative and there exists $x_0\in\R^N$ and a monotone function $v\in C^{\infty}(0,+\infty)$ such that $$u(x)=v(|x-x_0|),~~~~~~\forall~x\in\R^N.$$

$(4)$~~Let $N-2\leq \alpha<N$ if $N\geq3$ and $0<\al<N$ if $N=1,2$. If $u$ is a groundstate solution of \eqref{3.1}, then
$$\lim\limits_{|x|\rightarrow+\infty}|u(x)||x|^{\frac{N-1}{2}}e^{|x|}\in (0,+\infty).$$
Moreover, $|\nabla u(x)|=O(|x|^{-\frac{N-1}{2}}e^{-|x|})$ as $|x|\rightarrow+\infty.$
\end{lemma}

\begin{lemma}\label{lem3.2}~~(1)~~$d=\frac{c_*^2}{2}$.

(2)~~$u$ is a nontrivial solution of \eqref{3.1} with $|u|_2=c_*$ if and only if $u$ is a groundstate solution.
\end{lemma}
\begin{proof}~~For any nontrivial solution $u$ of \eqref{3.1}, then by Lemma \ref{lem3.1} (1)(2) and \eqref{3.4}, we have
$$c_*\leq |u|_2$$
and
$$
d\leq F(u)=\frac{1}{2}\int_{\R^N}|u|^2
$$
where equality holds only if $u$ is a groundstate solution. In particular, since $\widetilde{Q}_{\frac{N+\alpha+2}{N}}$ is a nontrivial solution of \eqref{3.1},
$$d\leq F(\widetilde{Q}_{\frac{N+\alpha+2}{N}})= \frac{|\widetilde{Q}_{\frac{N+\alpha+2}{N}}|_2^2}{2}=\frac{c_*^2}{2}.$$

Therefore, if $u$ is a groundstate solution of \eqref{3.1}, then by Lemma \ref{lem3.1} (3), $u$ is nontrivial and
$$\frac{c_*^2}{2}\leq \frac{|u|_2^2}{2}=F(u)=d\leq \frac{c_*^2}{2},$$
which shows that $d=\frac{c_*^2}{2}$ and $|u|_2=c_*$.

On the other hand, if $u$ is a nontrivial solution of \eqref{3.1} with $|u|_2=c_*$, then
$$\frac{ c_*^2}{2}=d\leq F(u)=\frac{1}{2}\int_{\R^N}|u|^2=\frac{ c_*^2}{2},$$
which implies that $F(u)=d$, i.e. $u$ is a groundstate solution.
\end{proof}

\begin{remark}\label{rem3.3}~~$\widetilde{Q}_{\frac{N+\alpha+2}{N}}$ is a groundstate solution of \eqref{3.1}.
\end{remark}

\begin{lemma}\label{lem3.4}(\cite{a})~~Suppose that $V\in L^{\infty}_{loc}(\R^N)$ and $\lim\limits_{|x|\rightarrow +\infty}V(x)=+\infty$, then the embedding $\mathcal{H}\hookrightarrow L^s(\R^N)$, $2\leq s<2^*$ is compact.
\end{lemma}

\noindent $\textbf{Proof of Theorem \ref{th1.2}}$\\
\begin{proof}~~Set $$C(u):=\ds\int_{\R^N}V(x)|u|^2\geq0,\ \ \ \forall~u\in H^1(\R^N),$$
then $$E(u)=\frac{A(u)}{2}+\frac{C(u)}{2}-\frac{N}{2(N+\alpha+2)}B(u).$$

(1)~~By \eqref{3.4}, for any $0<c\leq c_*$ and $u\in \widetilde{S}(c)$,
\begin{equation}\label{3.5}
E(u)\geq \frac{1}{2}\left[1-\left(\frac{c}{c_*}\right)^{\frac{2(\alpha+2)}{N}}\right]A(u)+\frac{1}{2}C(u)\geq0,
\end{equation}
then $e_c=\inf\limits_{u\in\widetilde{S}(c)}E(u)\geq0$ is well defined for $0<c\leq c_*$.

For each $0<c<c_*$, let $\{u_n\}\subset \widetilde{S}(c)$ be a minimizing sequence for $e_c$, then by \eqref{3.5}, $\{u_n\}$ is bounded in $\mathcal{H}$. Hence there exists $u_c\in \mathcal{H}$ such that $u_n\rightharpoonup u_c$ in $\mathcal{H}$. By Lemma \ref{lem3.4}, $u_n\rightarrow u_c$ in $L^s(\R^N)$, $2\leq s<2^*$, which implies that $|u_c|_2=c$ and $B(u_n)\rightarrow B(u_c)$. So $e_c\leq E(u_c)\leq \lim\limits_{n\rightarrow+\infty}E(u_n)=e_c$, i.e. $u_c\in \widetilde{S}(c)$ is a minimizer of $e_c$. Moreover, by \eqref{3.5}, $e_c>0$. So $e_c>0$ has at least one minimizer for all $0<c<c_*$.

(2)~~Let $N-2\leq \alpha<N$ if $N\geq3$ and $0<\al<N$ if $N=1,2$. For any $c>0$, let $\varphi\in C_0^{\infty}(\R^N)$ such that $0\leq \varphi(x)\leq1$, $\varphi(x)\equiv1$ for $|x|\leq1$, $\varphi(x)\equiv0$ for $|x|\geq2$ and $|\nabla \varphi|\leq 2$. For any $x_0\in\R^N$ and any $t>0$, set
\begin{equation}\label{3.6}
\widetilde{Q}^t(x)=\frac{cA_t t^{\frac{N}{2}}}{c_*}\varphi(x-x_0)\widetilde{Q}_{\frac{N+\alpha+2}{N}}(t(x-x_0)),
\end{equation}
where $A_t>0$ is chosen to satisfy that $|\widetilde{Q}^t|_2=c$. By the exponential decay of $\widetilde{Q}_{\frac{N+\alpha+2}{N}}$, we see that
$$\frac{1}{A_t^2}=1+\frac{1}{c_*^2}\ds\int_{\R^N}\left(\varphi^2(\frac{x}{t})-1\right)|\widetilde{Q}_{\frac{N+\alpha+2}{N}}(x)|^2
\rightarrow1$$
as $t\rightarrow +\infty$. Then $A_t$ depends only on $t$ and $\lim\limits_{t\rightarrow+\infty}A_t=1$. Since $V(x)\varphi^2(x-x_0)$ is bounded and has compact support, $C(\widetilde{Q}^t)\rightarrow \frac{c^2}{c_*^2}V(x_0)$.
$$\begin{array}{ll}
B(\widetilde{Q}^t)&=\ds(\frac{c A_t}{c_*})^{\frac{2(N+\alpha+2)}{N}}t^2\ds\left\{B(\widetilde{Q}_{\frac{N+\alpha+2}{N}})\right.\\[5mm]
& \left.+\ds\int_{\R^N}\{I_\alpha*[(|\varphi(\frac{y}{t})|^{\frac{N+\alpha+2}{N}}-1)|\widetilde{Q}_{\frac{N+\alpha+2}{N}}(y)|^{\frac{N+\alpha+2}{N}}]\}
(|\varphi(\frac{x}{t})|^{\frac{N+\alpha+2}{N}}+1)|\widetilde{Q}_{\frac{N+\alpha+2}{N}}(x)|^{\frac{N+\alpha+2}{N}}\right\}\\[5mm]
&:=\ds(\frac{c A_t}{c_*})^{\frac{2(N+\alpha+2)}{N}}t^2\left[B(\widetilde{Q}_{\frac{N+\alpha+2}{N}})+f_1(t)\right].
\end{array}$$
By the Hardy-Littlewood-Sobolev inequality \eqref{1.20} and the exponential decay of $\widetilde{Q}_{\frac{N+\alpha+2}{N}}$, we have there exists a constant $C>0$ such that
$$\begin{array}{ll}
|f_1(t)|&\leq C\left(\ds\int_{\R^N}|[\varphi(\frac{x}{t})]^{\frac{N+\alpha+2}{N}}-1|^{\frac{2N}{N+\alpha}}|\widetilde{Q}_{\frac{N+\alpha+2}{N}}(x)|^{\frac{2(N+\alpha+2)}{N+\alpha}}\right)^{\frac{N+\alpha}{2N}}\\[5mm]
&\leq C\left(\ds\int_{|x|\geq t}|\widetilde{Q}_{\frac{N+\alpha+2}{N}}(x)|^{\frac{2(N+\alpha+2)}{N+\alpha}}\right)^{\frac{N+\alpha}{2N}}\\[5mm]
&\leq C\left(\ds\int_{t}^{+\infty}r^{-\frac{2(N-1)}{N+\alpha}}e^{-\frac{2(N+\alpha+2)}{N+\alpha}r}\right)^{\frac{N+\alpha}{2N}}\leq Ct^{-\frac{2(N-1)}{2N}}e^{-\frac{N+\alpha+2}{N}t}\ \ \ \ \ \ \ \hbox{as}\ \ t\rightarrow+\infty.
\end{array}$$
Then by the exponential decay of $\widetilde{Q}_{\frac{N+\alpha+2}{N}}$ and $|\nabla \widetilde{Q}_{\frac{N+\alpha+2}{N}}|$, we have
\begin{equation}\label{3.7}
E(\widetilde{Q}^t)=\frac{c^2}{2c_*^2}t^2A(\widetilde{Q}_{\frac{N+\alpha+2}{N}})\left[1-\left(\frac{c}{c_*}\right)^{\frac{2(\alpha+2)}{N}}\right]+t^2f_2(t)+\frac{c^2}{2c_*^2}V(x_0)\ \ \hbox{as}\ \ t\rightarrow+\infty,
\end{equation}
where $f_2(t)$ denotes a function satisfying that $\lim\limits_{t\rightarrow+\infty}|f_2(t)|t^r=0$ for all $r>0$.

If $c>c^*$, then by \eqref{3.7}, $e_c\leq \lim\limits_{t\rightarrow+\infty}E(\widetilde{Q}^t)=-\infty$, hence $e_c=-\infty$ and there exists no minimizer for $e_c$.

If $c=c^*$, then by \eqref{3.5} and \eqref{3.7}, $0\leq e_{c_*}\leq \frac{V(x_0)}{2}.$
Taking the infimum over $x_0$, $e_{c_*}=0$. We just suppose that there exists $u\in \widetilde{S}(c_*)$ such that $E(u)=e_{c_*}$, then it follows from \eqref{3.5} that
\begin{equation}\label{3.9}
C(u)=0,
\end{equation}
which and the condition $(V_0)$ imply that $u$ must have compact support. On the other hand, $(E|_{\widetilde{S}(c_*)})'(u)=0$. Then there exists $\mu_{c_*}\in\R$ such that $E'(u)-\mu_{c_*}u=0$, i.e. for any $h\in C_0^{\infty}(\R^N)$,
\begin{equation}\label{3.10}\begin{array}{ll}
 0&=\langle E'(u)-\mu_{c_*}u,h\rangle\\[5mm]
 &=\ds\int_{\R^N}(\nabla u\nabla h-\mu_{c_*}uh)-\ds\int_{\R^N}(I_\alpha*|u|^{\frac{N+\alpha+2}{N}})|u|^{\frac{N+\alpha+2}{N}-2}uh\\[5mm]
 &=\langle I_{\frac{N+\alpha+2}{N}}'(u)-\mu_{c_*}u,h\rangle,
 \end{array}\end{equation}
where we have used the fact that $\int_{\R^N}V(x)uh=0$ due to the H\"{o}lder inequality and \eqref{3.9}. Then by Lemma \ref{lem2.5}, we see that $\mu_{c_*}<0$. Set $u(x):=(\sqrt{-\mu_{c_*}})^{\frac{N}{2}}w(\sqrt{-\mu_{c_*}}x)$, then by \eqref{3.10}, $w$ is a nontrivial solution of \eqref{3.1} with $|w|_2=c_*$, hence by Lemma \ref{lem3.2} $w$ is a groundstate solution. So by Lemma \ref{lem3.1} (4), $\lim\limits_{|x|\rightarrow +\infty}|u(x)||x|^{\frac{N-1}{2}}e^{|x|}\in(0,+\infty),$
which contradicts \eqref{3.9}.

Moreover, we conclude from \eqref{3.6} and \eqref{3.7} that $\limsup\limits_{c\rightarrow (c_*)^-}e_c\leq \frac{V(x_0)}{2}$ as $t\rightarrow+\infty$. By the arbitrary of $x_0$, we have $\lim\limits_{c\rightarrow (c_*)^-}e_c=0=e_{c_*}.$
\end{proof}

In the following, we consider the concentration behavior of minimizers as $c$ approaches $c_*$ from below when $N-2\leq \alpha<N$ if $N\geq3$ and $0<\al<N$ if $N=1,2$ and the potential $V(x)$ satisfies conditions $(V_0)(V_1)$.

\begin{lemma}\label{lem3.6}~~
Suppose that $(V_0)(V_1)$ hold, then there exist two positive constants $M_1<M_2$ independent of $c$ such that
$$M_1\left[1-\left(\frac{c}{c_*}\right)^{\frac{2(\alpha+2)}{N}}\right]^{\frac{q}{q+2}}\leq e_c\leq M_2\left[1-\left(\frac{c}{c_*}\right)^{\frac{2(\alpha+2)}{N}}\right]^{\frac{q}{q+2}}\ \ \ \ \hbox{as}\ \ c\rightarrow (c_*)^-,$$
where $q=\max\{q_1,q_2,\cdots,q_m\}$.
\end{lemma}
\begin{proof}~The proof consists of two steps.

\textbf{Step 1.}\ \ Without loss of generality, we may assume that $q=q_{i_0}$ for some $1\leq i_0\leq m$. By $(V_1)$, there exists $R>0$ small such that $V(x)\leq 2\mu_{i_0}|x-x_{i_0}|^{q_{i_0}}$ for $|x-x_{i_0}|\leq R$. Similarly to \eqref{3.6}, let
$$
u(x):=\frac{cA_{R,t} t^{\frac{N}{2}}}{c_*}\varphi\left(\frac{2(x-x_{i_0})}{R}\right)\widetilde{Q}_{\frac{N+\alpha+2}{N}}(t(x-x_{i_0}))\in \widetilde{S}(c),
$$
where $A_{R,t}>0$ and $A_{R,t}\rightarrow1$ as $t\rightarrow+\infty$. Then
$$C(u)\leq \frac{2\mu_{i_0}c^2A_{R,t}^2}{c_*^2}t^{-q_{i_0}}\ds\int_{\R^N}|x|^{q_{i_0}}|\widetilde{Q}_{\frac{N+\alpha+2}{N}}|^2.$$
Hence similarly to \eqref{3.7}, for large $t$,
$$e_c\leq \frac{A(\widetilde{Q}_{\frac{N+\alpha+2}{N}})}{2}t^2\left[1-\left(\frac{c}{c_*}\right)^{\frac{2(\alpha+2)}{N}}\right]
+2\mu_{i_0}t^{-q_{i_0}}\ds\int_{\R^N}|x|^{q_{i_0}}|\widetilde{Q}_{\frac{N+\alpha+2}{N}}(x)|^2+t^2h(t),$$
where $\lim\limits_{t\rightarrow+\infty}|h(t)|t^2=0$. By taking $t=[1-(\frac{c}{c_*})^{\frac{2(\alpha+2)}{N}}]^{-\frac{1}{q_{i_0}+2}},$ then there exists a constant $M_2>0$ independent of $c$ such that $$e_c\leq M_2 \left[1-\left(\frac{c}{c_*}\right)^{\frac{2(\alpha+2)}{N}}\right]^{\frac{q}{q+2}}.$$

\textbf{Step 2.}\ \ For any $0<c<c_*$, there exists $u_c\in \widetilde{S}(c)$ such that $E(u_c)=e_c$. By \eqref{3.5} and Theorem \ref{th1.2}, we see that
\begin{equation}\label{3.11}
C(u_c)\leq e_c\rightarrow 0 \ \ \ \hbox{as}\ \ c\rightarrow (c_*)^-.
\end{equation}
We claim that
\begin{equation}\label{3.12}
A(u_c)\rightarrow+\infty\ \ \ \hbox{as}\ \ c\rightarrow(c_*)^-.
\end{equation}
 In fact, by contradiction, if there exists a sequence $\{c_k\}\subset(0,c_*)$ with $c_k\rightarrow c_*$ as $k\rightarrow+\infty$ such that the sequence of minimizers $\{u_{c_k}\}\subset \widetilde{S}(c_k)$ is uniformly bounded in $\mathcal{H}$, then we may assume that for some $u\in\mathcal{H}$, $u_{c_k}\rightharpoonup u$ in $\mathcal{H}$ and by Lemma \ref{lem3.4} and \eqref{3.4},
$$ u_{c_k}\rightarrow u\ \ \ \hbox{in} \  L^2(\R^N)\ \ \ \hbox{and} \ \ \ B(u_{c_k})\rightarrow B(u).$$
Hence $u\in \widetilde{S}(c_*)$ and $0\leq e_{c_*}\leq E(u)\leq \lim\limits_{k\rightarrow+\infty}E(u_{c_k})=\lim\limits_{k\rightarrow+\infty}e_{c_k}=0,$
i.e. $u$ is a minimizer of $e_{c_*}$, which contradicts Theorem \ref{th1.2}.

Since
$$0\leq \frac{1}{2}A(u_c)-\frac{N}{2(N+\alpha+2)}B(u_c)\leq e_c,$$
 we see that
$$
\lim\limits_{c\rightarrow (c_*)^-}\frac{\frac{N}{N+\alpha+2}B(u_c)}{A(u_c)}=1.
$$
Then by \eqref{3.12}, set
\begin{equation}\label{3.13}
\varepsilon_c^{-2}:=\frac{N}{2(N+\alpha+2)}B(u_c)\rightarrow+\infty\ \ \ \hbox{as}\ \ c\rightarrow(c_*)^-
\end{equation}
and $\tilde{w}_c(x):=\varepsilon_c^{\frac{N}{2}}u_c(\varepsilon_cx).$ Then $|\tilde{w}_c|_2=c$ and
\begin{equation}\label{3.14}
\frac{N}{2(N+\alpha+2)}B(\tilde{w}_c)=1, \ \ \ \ \ \ \  2\leq A(\tilde{w}_c)\leq 2+2\varepsilon_c^2e_c.
\end{equation}
Let
$\delta:=\lim\limits_{c\rightarrow(c_*)^-}\sup\limits_{y\in\R^N}\int_{B_1(y)}|\tilde{w}_c|^2.$
If $\delta=0$, then $\tilde{w}_c\rightarrow 0$ in $L^s(\R^N)$ as $c\rightarrow(c_*)^-$, $2< s<2^*$, hence by \eqref{1.20}, $B(\tilde{w}_c)\rightarrow0$, which contradicts \eqref{3.14}. So $\delta>0$ and there exists $\{y_c\}\subset \R^N$ such that $\int_{B_1(y_c)}|\tilde{w}_c|^2\geq\frac{\delta}{2}>0.$ Set $$w_c(x):=\tilde{w}_c(x+y_c)=\varepsilon_c^{\frac{N}{2}}u_c(\varepsilon_cx+\varepsilon_cy_c),$$
 then \begin{equation}\label{3.15}
 \int_{B_1(0)}|w_c|^2\geq \frac{\delta}{2}>0.\end{equation}

 We claim that $\{\varepsilon_cy_c\}$ is uniformly bounded as $c\rightarrow (c_*)^-$. Indeed, if there exists a sequence $\{c_k\}\subset(0,c_*)$ with $c_k\rightarrow c_*$ as $k\rightarrow+\infty$ such that $|\varepsilon_{c_k}y_{c_k}|\rightarrow +\infty$ as $k\rightarrow+\infty$, then by $(V_0)$, \eqref{3.11} and \eqref{3.15} and the Fatou's Lemma, we have
 $$\begin{array}{ll}
0=\liminf\limits_{k\rightarrow+\infty}\ds\int_{\R^N}V(x)|u_{c_k}|^2&=\liminf\limits_{k\rightarrow+\infty}\ds\int_{\R^N}V(\varepsilon_{c_k}x+\varepsilon_{c_k}y_{c_k})|w_{c_k}(x)|^2\\[5mm]
 &\geq\ds\int_{\R^N}\liminf\limits_{k\rightarrow+\infty}[V(\varepsilon_{c_k}x+\varepsilon_{c_k}y_{c_k})|w_{c_k}(x)|^2]\\[5mm]
 &\geq\ds\int_{B_1(0)}\liminf\limits_{k\rightarrow+\infty}[V(\varepsilon_{c_k}x+\varepsilon_{c_k}y_{c_k})|w_{c_k}(x)|^2]\\[5mm]
 &\geq \ds(+\infty)\cdot\frac{\delta}{2}=+\infty,
 \end{array}$$
which is impossible. So $\{\varepsilon_cy_c\}$ is uniformly bounded as $c\rightarrow (c_*)^-$. Moreover, there exists $x_{j_0}\in\{x_1,\cdots,x_m\}$ such that
\begin{equation}\label{3.20}
\left\{\frac{\varepsilon_cy_c-x_{j_0}}{\varepsilon_c}\right\}\ \hbox{is~uniformly~bounded~as}~c\rightarrow(c_*)^-.
 \end{equation}

Indeed, by contradiction, we just suppose that for any $x_i\in\{x_1,\cdots,x_m\}$, there exists $c_k\rightarrow(c_*)^-$ as $k\rightarrow+\infty$ such that $|\frac{\varepsilon_{c_k}y_{c_k}-x_{i}}{\varepsilon_{c_k}}|\rightarrow+\infty$ as $k\rightarrow+\infty$. By $(V_1)$, \eqref{3.15} and the Fatou's Lemma, for any positive constant $C$,
$$
\begin{array}{ll}
\liminf\limits_{k\rightarrow+\infty}\varepsilon_{c_k}^{-q_i}\ds\int_{\R^N}V(\varepsilon_{c_k}x+\varepsilon_{c_k}y_{c_k})|w_{c_k}(x)|^2
&\geq\ds\int_{\R^N}\liminf\limits_{k\rightarrow+\infty}\frac{V(\varepsilon_{c_k}x+\varepsilon_{c_k}y_{c_k})}{\varepsilon_{c_k}^{q_i}}|w_{c_k}(x)|^2\\[5mm]
&\geq\ds\int_{\R^N}\liminf\limits_{k\rightarrow+\infty}\frac{V(\varepsilon_{c_k}x+x_{i})}{\varepsilon_{c_k}^{q_i}}|w_{c_k}(x+\frac{x_i-\varepsilon_{c_k}y_{c_k}}{\varepsilon_{c_k}})|^2\\[5mm]
&\geq\ds\mu_i\ds\int_{\R^N}\liminf\limits_{k\rightarrow+\infty}|x|^{q_i}|w_{c_k}(x+\frac{x_i-\varepsilon_{c_k}y_{c_k}}{\varepsilon_{c_k}})|^2\\[5mm]
&\geq\ds\mu_i\ds\int_{B_{1}(0)}\liminf\limits_{k\rightarrow+\infty}|x+\frac{\varepsilon_{c_k}y_{c_k}-x_i}{\varepsilon_{c_k}}|^{q_i}|w_{c_k}(x)|^2\\[5mm]
&\geq \ds\frac{\mu_i\delta}{2}C.
\end{array}
$$
Hence by \eqref{3.4} and \eqref{3.14},
\begin{equation}\label{3.21}
\begin{array}{ll}
e_{c_k}&=\ds\frac{1}{\varepsilon_{c_k}^2}\left(\frac{A(w_{c_k})}{2}-\frac{N B(w_{c_k})}{2(N+\alpha+2)}\right)
+\frac{1}{2}\ds\int_{\R^N}V(\varepsilon_{c_k}x+\varepsilon_{c_k}y_{c_k})|w_{c_k}|^2\\[5mm]
&\geq \ds\frac{1}{\varepsilon_{c_k}^2}\left[1-\left(\frac{c}{c_*}\right)^{\frac{2(\alpha+2)}{N}}\right]+\frac{\mu_i \delta C}{4}\varepsilon_{c_k}^{q_i}\\[5mm]
&\geq\ds (1+\frac{2}{q_i})\left(\frac{q_i\delta\mu_i}{8}\right)^{\frac{2}{q_i+2}}\left[1-\left(\frac{c_k}{c_*}\right)^{\frac{2(\alpha+2)}{N}}\right]^{\frac{q_i}{q_i+2}}C^{\frac{2}{q_i+2}}\\[5mm]
&\geq\ds (1+\frac{2}{q_i})\left(\frac{q_i\delta\mu_i}{8}\right)^{\frac{2}{q_i+2}}C^{\frac{2}{q_i+2}}\left[1-\left(\frac{c_k}{c_*}\right)^{\frac{2(\alpha+2)}{N}}\right]^{\frac{q}{q+2}}\ \ \ \hbox{as}\ \ k\rightarrow+\infty,
\end{array}\end{equation}
which contradicts the upper bound obtained in \textbf{Step 1} since $C>0$ is arbitrary. Then \eqref{3.20} holds. So for some $y_0\in\R^N,$
$$\frac{\varepsilon_cy_c-x_{j_0}}{\varepsilon_c}\rightarrow y_0\ \ \ \hbox{and}\ \ \ \varepsilon_cy_c\rightarrow x_{j_0}  \ \ \ \hbox{as}\  c\rightarrow(c_*)^-.$$
By the definition of $\{w_c\}$ and \eqref{3.14}, $\{w_c\}$ is uniformly bounded in $H^1(\R^N)$. Then up to a subsequence, we may assume that for some $w_0\in H^1(\R^N)$,
$$w_c\rightharpoonup w_0\ \ \hbox{in}\ H^1(\R^N),\ \ \ \ w_c\rightarrow w_0\ \ \hbox{in}\ L^s_{loc}(\R^N),\ 1\leq s<2^*$$
and $$w_c(x)\rightarrow w_0(x)\ \ \hbox{a.e.}\ \hbox{in}\ \R^N.$$
Then by $(V_1)$ and the Fatou's Lemma again, there exists a constant $C_2>0$ independent of $c$ such that
\begin{equation}\label{3.16}
\begin{array}{ll}
&\ \ \ \liminf\limits_{c\rightarrow(c_*)^-}\varepsilon_{c}^{-q_{j_0}}\ds\int_{\R^N}V(\varepsilon_{c}x+\varepsilon_{c}y_{c})|w_{c}(x)|^2\\[5mm]
&\geq \ds\int_{\R^N}\liminf\limits_{c\rightarrow(c_*)^-}\frac{V(\varepsilon_{c}x+\varepsilon_{c}y_{c})}{|\varepsilon_{c}|^{q_{j_0}}}|w_{c}(x)|^2\\[5mm]
&\geq \ds\int_{\R^N}\liminf\limits_{c\rightarrow(c_*)^-}\frac{V(\varepsilon_{c}x+\varepsilon_{c}y_{c})}{|\varepsilon_{c}x+\varepsilon_{c}y_{c}-x_{j_0}|^{q_{j_0}}}|x+\frac{\varepsilon_{c}y_{c}-x_{j_0}}{\varepsilon_c}|^{q_{j_0}}|w_{c}(x)|^2\\[5mm]
&\geq \mu_{j_0}\ds\int_{B_1(0)}|x+y_0|^{q_{j_0}}|w_{0}(x)|^2:=C_2>0.
\end{array}
\end{equation}
Similarly to \eqref{3.21}, we have
$$e_c\geq(1+\frac{2}{q_{j_0}})\left(\frac{q_{j_0}C_2}{2}\right)^{\frac{2}{q_{j_0}+2}}\left[1-\left(\frac{c}{c_*}\right)^{\frac{2(\alpha+2)}{N}}\right]^{\frac{q}{q+2}}:=M_1\left[1-\left(\frac{c}{c_*}\right)^{\frac{2(\alpha+2)}{N}}\right]^{\frac{q}{q+2}}$$
as $c\rightarrow(c_*)^-.$
\end{proof}

\begin{lemma}\label{lem3.7}~~Suppose that $u_c$ is a minimizer of $e_c$ and $V(x)$ satisfies $(V_0)(V_1)$, then there exist two positive constants $K_1<K_2$ independent of $c$ such that
$$K_1\left[1-\left(\frac{c}{c_*}\right)^{\frac{2(\alpha+2)}{N}}\right]^{-\frac{2}{q+2}}\leq A(u_c)\leq K_2\left[1-\left(\frac{c}{c_*}\right)^{\frac{2(\alpha+2)}{N}}\right]^{-\frac{2}{q+2}}\ \ \ \hbox{as}\ c\rightarrow(c_*)^-.$$
\end{lemma}
\begin{proof}~~The idea of the proof comes from that of Lemma 4 in \cite{gs}, but it needs more careful analysis.

By \eqref{3.5}, we see that $$e_c=E(u_c)\geq\frac{1}{2}\left[1-(\frac{c}{c_*})^{\frac{2(\alpha+2)}{N}}\right]A(u_c),$$
 then by Lemma \ref{lem3.6},
 $$A(u_c)\leq 2M_2\left[1-\left(\frac{c}{c_*}\right)^{\frac{2(\alpha+2)}{N}}\right]^{-\frac{2}{q+2}},$$ where $M_2$ is given in Lemma \ref{lem3.6}.

 For any fixed $b\in(0,c)$, there exist two functions $u_b\in \widetilde{S}(b)$, $u_c\in \widetilde{S}(c)$ such that $e_b=E(u_b)$ and $e_c=E(u_c)$ respectively. Then by \eqref{3.4}, we see that
$$e_b\leq E\left(\frac{b}{c}u_c\right)< e_c+\frac{1}{2}\left[1-\left(\frac{b}{c}\right)^{\frac{2(\alpha+2)}{N}}\right]A(u_c).$$
Let $\eta:=\frac{c-b}{c_*-c}>0$, then $\eta\rightarrow +\infty$ as $c\rightarrow(c_*)^-$. Then by Lemma \ref{3.6}, we have
$$\begin{array}{ll}
\ds\frac{1}{2}A(u_c)&>\ds\frac{e_b-e_c}{1-(\frac{b}{c})^{\frac{2(\alpha+2)}{N}}}\\[5mm]
&\geq\ds\frac{M_1(1-(\frac{b}{c_*})^{\frac{2(\alpha+2)}{N}})^{\frac{q}{q+2}}-M_2(1-(\frac{c}{c_*})^{\frac{2(\alpha+2)}{N}})^{\frac{q}{q+2}}}
{1-(\frac{b}{c})^{\frac{2(\alpha+2)}{N}}}\\[5mm]
&\geq \ds\left[1-\left(\frac{c}{c_*}\right)^{\frac{2(\alpha+2)}{N}}\right]^{-\frac{2}{q+2}}
\ds\frac{M_1\ds[\frac{1-(\frac{b}{c_*})^{\frac{2(\alpha+2)}{N}}}{1-(\frac{c}{c_*})^{\frac{2(\alpha+2)}{N}}}]^{\frac{q}{q+2}}-M_2}{(1-(\frac{b}{c})^{\frac{2(\alpha+2)}{N}})[1-(\frac{c}{c_*})^{\frac{2(\alpha+2)}{N}}]^{-1}}\\[5mm]
&\geq\ds\left[1-\left(\frac{c}{c_*}\right)^{\frac{2(\alpha+2)}{N}}\right]^{-\frac{2}{q+2}}
\ds\frac{M_1(\frac{N}{2(\alpha+2)})^{\frac{q}{q+2}}(1+\eta)^{\frac{q}{q+2}}-M_2}{\eta},
\end{array}$$
which gives the desired positive lower bound as $c\rightarrow(c_*)^-.$
\end{proof}

\noindent $\textbf{Proof of Theorem \ref{th1.3}}$\\
\begin{proof}~~Let $\{c_k\}\subset(0,c_*)$ be a sequence satisfying $c_k\rightarrow (c_*)^-$ as $k\rightarrow+\infty$ and denote $\{u_{c_k}\}\subset \widetilde{S}(c_k)$ to be a sequence of minimizers for $e_{c_k}$. Set
\begin{equation}\label{3.22}
\varepsilon_k:=[1-(\frac{c_k}{c_*})^{\frac{2(\alpha+2)}{N}}]^{\frac{1}{q+2}}>0.
\end{equation}
By \eqref{3.5}, Lemmas \ref{lem3.6} and \ref{lem3.7}, we see that
$$
K_1\varepsilon_k^{-2}\leq A(u_{c_k})\leq K_2\varepsilon_k^{-2},\ \ \ 0\leq C(u_{c_k})\leq 2M_2\varepsilon_k^q
$$
Let $$\tilde{w}_{c_k}(x):=\varepsilon_k^{\frac{N}{2}}u_{c_k}(\varepsilon_k x),$$
then $|\tilde{w}_c|_2=c$ and
\begin{equation}\label{3.23}
K_1\leq A(\tilde{w}_{c_k})\leq K_2, \ \ \ \ \ B(\tilde{w}_{c_k})\leq \frac{N+\alpha+2}{N}K_2
\end{equation}

Let $\delta:=\lim\limits_{k\rightarrow+\infty}\sup\limits_{y\in\R^N}\int_{B_1(0)}|\tilde{w}_{c_k}|^2.$ If $\delta=0$, then by the Vanishing Lemma \ref{lem2.6}, $\tilde{w}_{c_k}\rightarrow0$ in $L^s(\R^N)$ as $k\rightarrow+\infty$, $2<s<2^*$. Hence by \eqref{1.4}, $B(\tilde{w}_{c_k})\rightarrow0$. So
$$0<\frac{K_1}{2}\leq \frac{A(\tilde{w}_{c_k})}{2}\leq e_{c_k}\varepsilon_k^2+\frac{N}{2(N+\alpha+2)}B(\tilde{w}_{c_k})\rightarrow0\ \ \hbox{as}\ k\rightarrow+\infty,$$
 which is a contradiction. Then $\delta>0$ and there exists $\{y_k\}\subset\R^N$ such that
 $\int_{B_1(y_k)}|w_{c_k}|^2\geq\frac{\delta}{2}>0.$ Set $$w_{c_k}(x):=\tilde{w}_{c_k}(x+y_{k})=\varepsilon_k^{\frac{N}{2}}u_{c_k}(\varepsilon_k x+\varepsilon y_k), $$
 then
 \begin{equation}\label{3.24}
 \int_{B_1(0)}|w_{c_k}|^2\geq\frac{\delta}{2}>0
 \end{equation}
 and
 \begin{equation}\label{3.25}
\ds\int_{\R^N}V(\varepsilon_k x+\varepsilon_k y_k)|w_{c_k}(x)|^2=C(u_{c_k})\leq 2M_2\varepsilon_k^q.
\end{equation}
Similar to the proof in Lemma \ref{lem3.6}, one can show that $\{\varepsilon_k y_k\}$ is uniformly bounded as $k\rightarrow+\infty$.

Since $u_{c_k}\in \widetilde{S}(c_k)$ is a minimizer of $e_{c_k}$, $(E|_{\widetilde{S}(c_k)})'(u_{c_k})=0$, i.e. there exists a sequence $\{\lambda_k\}\subset\R$ such that $E'(u_{c_k})-\lambda_ku_{c_k}=0$ in $\mathcal{H}^{-1}$, where $\mathcal{H}^{-1}$ denotes the dual space of $\mathcal{H}$. Then
$$\varepsilon_k^2\lambda_k
=\frac{2\frac{N+\alpha+2}{N}\varepsilon_k^2e_{c_k}-\frac{\alpha+2}{N}\varepsilon_k^2 C(u_{c_k})-\frac{\alpha+2}{N}A(w_{c_k})}{c^2_k},$$
which and \eqref{3.23}\eqref{3.25} imply that there exists $\beta>0$ such that
$$\varepsilon_k^2\lambda_k\rightarrow-\beta^2\ \ \hbox{as}\ k\rightarrow+\infty.$$
By the definition of $w_{c_k}$, we see that $w_{c_k}$ satisfies the following equation
\begin{equation}\label{3.26}
-\Delta w_{c_k}+\varepsilon_k^2V(\varepsilon_k x+\varepsilon_k y_k)w_{c_k}-(I_\alpha*|w_{c_k}|^{\frac{N+\alpha+2}{N}})|w_{c_k}|^{\frac{N+\alpha+2}{N}-2}w_{c_k}
=\lambda_k\varepsilon_k^2 w_{c_k}\ \ \hbox{in}\ \R^N.
\end{equation}
Since $\{w_{c_k}\}$ is uniformly bounded in $H^1(\R^N)$, there exists $w_0\in H^1(\R^N)$ such that
$$w_{c_k}\rightharpoonup w_0\ \ \hbox{in}\ H^1(\R^N),\ \  \ \ \ \ \ w_{c_k}\rightarrow w_0\ \ \hbox{in}\ L^s_{loc}(\R^N),\ 1\leq s<2^*$$
and
$$w_{c_k}(x)\rightarrow w_0(x)\ \ \ \hbox{a.e.~in}\ \R^N.$$
Moreover, \eqref{3.24} implies that $w_0\neq0$. Then $w_0$ is a nontrivial solution of $-\Delta w_0+\beta^2 w_0=(I_\alpha*|w_0|^{\frac{N+\alpha+2}{N}})|w_0|^{\frac{N+\alpha+2}{N}-2}w_0$ in $\R^N.$ Set $$w_0(x):=\beta^{\frac{N}{2}}W_0(\beta x),$$ then $W_0$ is a nontrivial solution of
\begin{equation}\label{3.27}
-\Delta W_0+W_0=(I_\alpha*|W_0|^{\frac{N+\alpha+2}{N}})|W_0|^{\frac{N+\alpha+2}{N}-2}W_0,\ \  \ x\in\R^N.
\end{equation}
Hence by Lemma \ref{lem3.1} (2), we have $A(W_0)=\frac{N}{N+\alpha+2}B(W_0).$ So it follows from \eqref{3.4} that
$$
c_*^{\frac{2(\alpha+2)}{N}}\leq \frac{\frac{N+\alpha+2}{N}A(W_0)|W_0|_2^{\frac{2(\alpha+2)}{N}}}{B(W_0)}=|W_0|_2^{\frac{2(\alpha+2)}{N}}=|w_0|_2^{\frac{2(\alpha+2)}{N}}\leq \lim\limits_{k\rightarrow+\infty}|w_{c_k}|_2^{\frac{2(\alpha+2)}{N}}=c_*^{\frac{2(\alpha+2)}{N}},
$$
i.e. $|w_0|_2=|W_0|_2=c_*$. Hence $w_{c_k}\rightarrow w_0$ in $L^2(\R^N)$ and then by the interpolation inequality,
$$w_{c_k}\rightarrow w_0\ \ \hbox{in}\ L^s(\R^N)\ \  \hbox{for\ all}\ \ 2\leq s<2^*.$$
Moreover, Lemma \ref{lem3.2} shows that $W_0$ is a groundstate solution of \eqref{3.27}.  So by Lemma \ref{lem3.1} (3)(4), $W_0(x)=O(|x|^{-\frac{N-1}{2}}e^{-|x|})$ as $|x|\rightarrow+\infty$ and we may assume that up to translations, $W_0(x)$ is radially symmetric about the origin.

By \eqref{3.25}, we see that for any $q_i\in\{q_1,\cdots,q_m\}$,
$$\frac{1}{\varepsilon^{q_i}}\ds\int_{\R^N}V(\varepsilon_k x+\varepsilon_ky_k)|w_{c_k}(x)|^2\leq 2M_2.$$
Similarly to the proof of \eqref{3.20}, there exists $x_{j_0}\in\{x_1,\cdots,x_m\}$ and $y_0\in\R^N$ such that $\varepsilon_k y_k\rightarrow x_{j_0}$ and $\frac{\varepsilon_k y_k-x_{j_0}}{\varepsilon_k}\rightarrow y_0$ as $k\rightarrow +\infty$.
Then similarly to \eqref{3.16}, we see that
\begin{equation}\label{3.28}\begin{array}{ll}
\ds\liminf\limits_{k\rightarrow +\infty}\frac{1}{\varepsilon_k^{q}}\ds\int_{\R^N}V(\varepsilon_kx+\varepsilon_ky_k)|w_{c_k}(x)|^2
&\geq\ds\int_{\R^N}\ds\liminf\limits_{k\rightarrow +\infty}\frac{ V(\varepsilon_kx+\varepsilon_ky_k)}{\varepsilon_k^{q_{j_0}}}|w_{c_k}(x)|^2\\[5mm]
&\geq\mu_{j_0}\ds\int_{\R^N}|x+y_0|^{q_{j_0}}|w_0(x)|^2\\[5mm]
&=\ds\frac{\mu_{j_0}}{\beta^{q_{j_0}}}\ds\int_{\R^N}|x+\beta y_0|^{q_{j_0}}|W_0(|x|)|^2\\[5mm]
&\geq \ds\frac{\mu_{j_0}}{\beta^{q_{j_0}}}\ds\int_{\R^N}|x|^{q_{j_0}}|W_0(x)|^2,
\end{array}\end{equation}
where the last inequality is strict if and only if $y_0\neq0$. Hence similarly to \eqref{3.21},
\begin{equation}\label{3.29}\begin{array}{ll}
\ds\liminf\limits_{k\rightarrow +\infty}\frac{e_{c_k}}{\varepsilon_k^{q}}
&\geq \ds\frac{1}{2}A(w_0)+\liminf\limits_{k\rightarrow +\infty}\ds\frac{1}{2\varepsilon_k^q}\ds\int_{\R^N}V(\varepsilon_k x+\varepsilon_ky_k)|w_{c_k}(x)|^2\\[5mm]
&\geq\ds\frac{1}{2}\left(\beta^2c_*^2\frac{N}{\alpha+2}+\ds\frac{\mu_{j_0}}{\beta^{q_{j_0}}}\ds\int_{\R^N}|x|^{q_{j_0}}|W_0(x)|^2\right)\\[5mm]
&\geq \ds c_*^2\left(\frac{\beta^2}{2}\frac{N}{\alpha+2}+\frac{\lambda_{j_0}^{q_{j_0}+2}}{q_{j_0}\beta^{q_{j_0}}}\right)\\[5mm]
&\geq\ds\frac{\lambda_{j_0}^2c_*^2}{2}\left(\frac{ N}{\alpha+2}\right)^{\frac{q_{j_0}}{q_{j_0}+2}}\frac{q_{j_0}+2}{q_{j_0}}\\[5mm]
&\geq\ds\frac{\lambda^2c_*^2}{2}\left(\frac{N}{\alpha+2}\right)^{\frac{q}{q+2}}\frac{q+2}{q},
\end{array}
\end{equation}
where $\lambda=\min\limits_{1\leq i\leq m}\lambda_i$ and $q=\max\limits_{1\leq i\leq m}q_i$.

On the other hand, for any $x_i\in \{x_1,\cdots,x_m\}$ and $t>0$, let $v_k(x)=A_k\frac{c_k}{c_*}\left(\frac{t}{\varepsilon_k}\right)^{\frac{N}{2}}\varphi(x-x_i)W_0(\frac{t(x-x_i)}{\varepsilon_k})$, where $\varphi$ is a cut-off function given as in \eqref{3.6} and $A_k>0$ is chosen to satisfy that $v_k\in \widetilde{S}(c_k)$. Then $A_k\rightarrow1$ as $k\rightarrow+\infty$. Similarly to \eqref{3.7}, by the Dominated Convergence theorem, we see that
\begin{equation}\label{3.30}\begin{array}{ll}
\ds\lim\limits_{k\rightarrow +\infty}\frac{E(v_k)}{\varepsilon_k^{q}}&=\ds\frac{t^2}{2}
A(W_0)+\lim\limits_{k\rightarrow +\infty}\frac{t^N}{2\varepsilon_k^{N+q}}\ds\int_{\R^N}V(x)|\varphi(x-x_i)W_0(\frac{t(x-x_i)}{\varepsilon_k})|^2\\[5mm]
&=\ds\frac12\left(\frac{t^2c_*^2 N}{\alpha+2}
+\frac{\overline{\mu}_i}{t^{q}}\ds\int_{\R^N}|x|^{q}|W_0(x)|^2\right)\\[5mm]
&\leq\ds c_*^2\left(\frac{t^2}{2}\frac{N}{\alpha+2}
+\frac{\overline{\lambda}_i^{q+2}}{qt^{q}}\right),\end{array}
\end{equation}
where $$\overline{\mu}_i=\lim\limits_{x\rightarrow x_i}\frac{V(x)}{|x-x_i|^q}=\left\{%
\begin{array}{ll}
    \mu_i, & \hbox{if~$q=q_i$}, \\
   +\infty, & \hbox{if~$q\neq q_i$}\\
\end{array}%
\right.$$ and $$\overline{\lambda}_i=\left(\frac{\overline{\mu}_i q}{2c_*^2}\int_{\R^N}|x|^{q}|W_0(x)|^2\right)^{\frac{1}{q+2}}=\left\{%
\begin{array}{ll}
    \lambda_i, & \hbox{if~$q=q_i$}, \\
   +\infty, & \hbox{if~$q\neq q_i$}\\
\end{array}%
\right..$$
 So, since $t>0$ is arbitrary, by taking the infimum over $\{\overline{\lambda}_i\}_{i=1}^{m}$ in \eqref{3.30} and combining \eqref{3.29}, we see that
 $$\ds\lim\limits_{k\rightarrow+\infty}\frac{e_{c_k}}{\varepsilon_k^{q}}= \ds\frac{\lambda^2c_*^2}{2}\left(\frac{N}{\alpha+2}\right)^{\frac{q}{q+2}}\frac{q+2}{q}.$$
Then \eqref{3.28}-\eqref{3.30} must be equalities, which imply that
$$y_0=0,\ \ \ \beta=\left(\frac{\alpha+2}{N}\right)^{\frac{1}{q+2}}\lambda$$
and $\varepsilon_k y_k\rightarrow x_{j_0}\in\{x_i|~\lambda_i=\lambda,~1\leq i\leq m\}$. Therefore,
$$\varepsilon_k^{\frac{N}{2}}u_{c_k}(\varepsilon_k x+\varepsilon_k y_k)=w_{c_k}(x)\rightarrow w_0(x)= \left((\frac{\alpha+2}{N})^{\frac{1}{q+2}}\lambda\right)^{\frac{N}{2}} W_0((\frac{\alpha+2}{N})^{\frac{1}{q+2}}\lambda x)$$
in $L^{\frac{2Ns}{N+\alpha}}(\R^N)$ for all $\frac{N+\alpha}{N}\leq s<\frac{N+\alpha}{(N-2)_+}$.

\end{proof}

 \end{document}